\documentclass[11pt,a4paper,leqno]{amsart}
\usepackage{a4wide,verbatim}

\title[Geodesic Gauss map.]{Minimal Lagrangian submanifolds via the geodesic Gauss map.}
\author{Chris Draper}
\author{Ian McIntosh}
\address{Department of Mathematics\\ University of York\\ York YO10 5DD, UK}
\email{ian.mcintosh@york.ac.uk}
\subjclass{53C42, 53D12}
\date{27 April, 2015}

\newcommand{\R}{\mathbb{R}}

\newcommand{\Gr}{\mathrm{Gr}}

\newcommand{\so}{\mathfrak{so}}

\newcommand{\fg}{\mathfrak{g}}
\newcommand{\fh}{\mathfrak{h}}
\newcommand{\fk}{\mathfrak{k}}
\newcommand{\fl}{\mathfrak{l}}
\newcommand{\fm}{\mathfrak{m}}
\newcommand{\fn}{\mathfrak{n}}

\newcommand{\fp}{\mathfrak{p}}
\newcommand{\fq}{\mathfrak{q}}
\newcommand{\fs}{\mathfrak{s}}

\newcommand{\fv}{\mathfrak{v}}
\newcommand{\fum}{\mathfrak{um}}

\newcommand{\caA}{\mathcal{A}}

\newcommand{\caC}{\mathcal{C}}

\newcommand{\caH}{\mathcal{H}}

\newcommand{\caL}{\mathcal{L}}

\newcommand{\caQ}{\mathcal{Q}}

\newcommand{\caS}{\mathcal{S}}

\newcommand{\caV}{\mathcal{V}}
\newcommand{\caZ}{\mathcal{Z}}

\newcommand{\End}{\operatorname{End}}
\newcommand{\Hom}{\operatorname{Hom}}
\newcommand{\Sp}{\operatorname{Sp}}

\newcommand{\tr}{\operatorname{tr}}

\newcommand{\g}[2]{\langle{#1},{#2}\rangle}

\newcommand{\Kah}{K\"{a}hler}

\newcommand{\Ad}{\operatorname{Ad}}
\newcommand{\ad}{\operatorname{ad}}
\newcommand{\II}{\mathrm{I\! I}}

\renewcommand{\implies}{\Rightarrow}

\newtheorem{thm}{Theorem}[section]
\newtheorem{prop}[thm]{Proposition}

\newtheorem{lem}[thm]{Lemma}

\newtheorem{defn}[thm]{Definition}
\theoremstyle{remark}
\newtheorem{exam}[thm]{Example}
\newtheorem{rem}[thm]{Remark}

\numberwithin{equation}{section}

\begin{document}

\begin{abstract}
For an oriented isometric immersion $f:M\to S^n$ the spherical Gauss map is the Legendrian immersion of its unit
normal bundle $UM^\perp$ into the unit sphere subbundle of $TS^n$, and the geodesic Gauss map $\gamma$ projects this
into the manifold of oriented geodesics in $S^n$ (the Grassmannian of oriented $2$-planes in $\mathbb{R}^{n+1}$), 
giving a Lagrangian immersion of $UM^\perp$ into a K\"{a}hler-Einstein manifold. 
We give expressions for the mean curvature vectors for both the spherical and geodesic Gauss maps in terms of the
second fundamental form of $f$, and show that when $f$ has conformal shape form this depends only on the mean curvature
of $f$. In particular we deduce that the geodesic Gauss map of every minimal surface in $S^n$ is minimal Lagrangian. We
also give simple proofs that: deformations of $f$ always correspond to Hamiltonian deformations of $\gamma$; the mean curvature
vector of $\gamma$ is always a Hamiltonian vector field. This extends work of Palmer on the case when $M$ is a
hypersurface.
\end{abstract}

\maketitle

\section{Introduction.}
A well-known example in symplectic geometry says that, when $f:M\to (N,g)$ is an immersion of a manifold 
into a Riemannian manifold $(N,g)$,
its normal bundle $TM^\perp\to TN$ is a Lagrangian immersion when $TN$ is equipped with its canonical symplectic
structure. Moreover, the unit sphere bundle $UN\subset TN$ has a contact structure and the unit normal bundle of $M$ provides 
a Legendrian immersion
$\mu:UM^\perp\to UN$. This map is sometimes called the \emph{spherical Gauss map}.
The Reeb vector field for this contact structure is the geodesic flow. When this flow is generated by a one-parameter
group action (for example, when every geodesic is closed and has the same length, as is the case with compact rank one
symmetric spaces) the quotient by this action is the manifold of oriented geodesics in $N$, which we will denote by
$\caQ$. For example, when $N=S^n$ this manifold of geodesics is isomorphic to the Grassmanian $\Gr(2,n+1)$ of oriented
$2$-planes in $\R^{n+1}$. Taking this quotient is an example of symplectic reduction of $TN$, and the Legendrian immersion
$\mu$ projects to a Lagrangian immersion $\gamma:UM^\perp\to\caQ$. We will call $\gamma$ the \emph{geodesic Gauss map} for
$f:M\to N$. We will always assume $f(M)$ is orientable, in which case its geodesic Gauss map assigns to each unit normal 
vector the oriented geodesic it generates. When $N=S^n$ and $M$ is a hypersurface, this is just the usual Gauss map 
for a codimension two submanifold of $\R^{n+1}$.

Now it is interesting to ask what the relationship is between the Riemannian properties of 
$f:M\to N$ and $\gamma:UM^\perp\to \caQ$ as isometric immersions. For this one needs a metric on $\caQ$. 
In this article we examine the most natural case, where
$N=S^n$ (with the metric of constant curvature $1$) and $TS^n$ is given the Sasaki metric. In this case geodesic flow 
preserves the metric, which descends to the
standard \Kah-Einstein metric on $\caQ\simeq\Gr(2,n+1)$. These correspondences are most easily understood using
homogeneous geometry, since $S^n$, $US^n$ and $\caQ$ are all reductive homogeneous spaces of $SO(n+1)$ and the metric for
each is just the normal metric determined by a choice of adjoint-invariant bilinear form on $\so(n+1)$. Given this, we
derive expressions for the mean curvature vectors $H_\mu$ and $H_\gamma$ of $\mu$ and $\gamma$
in terms of the second fundamental form $\II_f$ of $f$. 
In fact, we first show that $H_\gamma = (\pi_\caQ)_*H_\mu$ for the projection $\pi_\caQ:US^n\to\caQ$, so
that we need only do calculations for $\mu$. 
When $M$ is a hypersurface the resulting expression for $H_\gamma$ gives another way of looking at
Palmer's formula \cite{Pal97} for $H_\gamma$, which he wrote in terms of the principal curvatures of $f$. 

The relationship between $H_\mu$ and $\II_f$ is mediated by the normal bundle
projection $\pi^\perp:UM^\perp\to M$, and it is only when this is conformal that one obtains
a direct relationship between $H_\mu$ and $H_f = \tr_g\II_f$. In fact $\pi^\perp$ is conformal if and only if $f$ has
\emph{conformal shape form}, by which we mean that its shape operator $A_f:TM^\perp\to\End(TM)$ satisfies $A_f(\xi)^2 =
r(\xi)^2I$ for every $\xi\in UM^\perp$ and some smooth function $r:UM^\perp\to\R$ (cf.\ \cite{GudM}, where they refer to
this property by the name \emph{conformal second fundamental form}). Under this condition we prove (Theorem \ref{thm:main}
below) that:
\begin{enumerate}
\item when $M$ has codimension two or more, $\gamma$ is minimal Lagrangian if and only if $f$ is minimal,
\item when $M$ is a hypersurface, $\gamma$ is minimal Lagrangian if and only if $f$ has constant (possibly zero) mean
curvature.
\end{enumerate}
The condition that $f$ have conformal shape form is somewhat restrictive, but still allows large families of examples. It
is easy to see that every minimal surface has conformal shape form, so these produce extensive families of
minimal Lagrangian submanifolds in $\caQ$ for $n\geq 2$: when $n=2$ this had already been observed by Palmer
\cite{Pal94} (see also \cite{CasU}).
Palmer's more general formula \cite{Pal97} already makes it clear that any isoparametric hypersurface  
of $S^n$ provides a minimal Lagrangian submanifold of $\Gr(2,n+1)$. 
For hypersurfaces of dimension three or more, the combined conditions of conformal shape form and constant
mean curvature force the hypersurface to be isoparametric with at most two
distinct principal curvatures, so our correspondence provides no new examples for hypersurfaces.

In the final section we discuss Lagrangian and Hamiltonian deformations of the geodesic Gauss map. 
We prove two results using
simple circle bundle arguments which exploit the fact that the pullback $\gamma^{-1}US^n$ is a flat circle bundle for
which $\mu$ is a horizontal section.  
Any deformation of $f$ which preserves the diffeomorphism class of the unit normal bundle corresponds
to a (not necessarily unique) deformation of $\gamma$ through Lagrangian immersions,
and we prove that these are always Hamiltonian deformations. For hypersurfaces this was proved in \cite{MaO09}, where the
converse (that short term Hamiltonian deformations arise from deformations of the hypersurface) was also proved. 
This converse cannot be expected to hold for higher codimension immersions $f$ and we explain why. 
We also prove that the mean curvature variation $H_\gamma$ is always a
Hamiltonian variation (for hypersurfaces Palmer \cite{Pal97} deduced this directly from his formula for $H_\gamma$). 
This means it is reasonable
to ask when the Lagrangian mean curvature flow of $\gamma$ will keep it inside the class of geodesic Gauss maps of
immersions of $M$. At present we have no answer to this question.

\medskip\noindent
\textit{Notation.} Throughout this article we will set $G=SO(n+1)$, with Lie algebra $\fg=\so(n+1)$, and equip the latter with the
$\Ad_G$-invariant inner product $\g{\eta}{\zeta} = -\tfrac12 \tr(\eta\zeta)$. For any closed Lie subgroup $K\subset G$,
the quotient $G/K$ is a reductive homogeneous space with reductive decomposition $\fg = \fk+\fm$ (where $\fk$ the
Lie algebra of $K$ and $\fm=\fk^\perp$).  If $\fv\subset\fg$ is a $K$-invariant 
subspace for the adjoint action of $K$, we will denote by $[\fv]_K$ the subbundle
\[
[\fv]_K =\{(gK,\Ad g\cdot\eta):\eta\in\fv\}\subset G/K\times \fg.
\]
In particular, the tangent space $T(G/K)$ is canonically isomorphic to $[\fm]_K$ via 
\[
\beta_K: T(G/K)\to [\fm]_K;\quad \beta_K(X_p) = (p,\xi),\ X_p = \frac{d}{dt}(e^{t\xi}\cdot p)_0.
\]

\section{The geometry of $US^n$.}

By fixing the base point $o=(0,\ldots,0,1)\in S^n\subset\R^{n+1}$ we view $S^n$ as the orbit of $o$ by the standard action
of $G$. This fixes an isomorphism $S^n\simeq G/K$, with isotropy group $K\simeq SO(n)$. 
The inner product induces the round metric $g$ on $S^n$ of sectional curvature $1$.
Therefore $US^n\simeq [\fum]_K$ where $\fum$ is the unit sphere in $\fm$. 

Clearly, $K$ acts transitively on $\fum$ (this is true for any compact rank one symmetric space) and therefore $G$ acts 
transitively on $US^n\simeq[\fum]_K$ by \begin{equation}\label{eq:Gaction}
G\times [\fum]_K\to[\fum]_K;\quad g\cdot(aK,\eta) = (gaK,\Ad g\cdot\eta). 
\end{equation}
We fix a base point $(o,\nu_0)$ for this action by taking
\[
\nu_0 = \begin{pmatrix} 0 & \ldots & 0 & 0\\ \vdots & & & \vdots\\ 0&\ldots &0 & 1\\  0&\ldots & -1 & 0\end{pmatrix}\in\fum.
\]
The isotropy group is therefore
\[
H = \{g\in K:\Ad\cdot\nu_0 = \nu_0\}\simeq SO(n-1).
\]
This fixes an isomorphism $US^n\simeq G/H$. We write the corresponding reductive decomposition as
$\fg = \fh + \fp$ for $\fp=\fh^\perp$, so that $\beta_H:TUS^n\to [\fp]_H$ is the canonical isomorphism. 
We will equip $G/H$ with the $G$-invariant normal metric $h$ induced by the inner product $\g{\ }{\ }$. 
It follows that $\pi:(US^n,h)\to (S^n,g)$ is a $G$-equivariant Riemannian submersion with totally geodesic fibres. 
In fact we will show below that $h$ is the restriction of the Sasaki metric on $TS^n$ to $US^n$. We will denote the horizontal
and vertical decomposition of $TUS^n$ with respect to $\pi:US^n\to S^n$ by $TUS^n=\caH+\caV$. It follows that $\beta_H$
produces isomorphisms $\caH\simeq [\fm]_H$, $\caV\simeq [\fn]_H$, where $\fn = \fm^\perp\cap\fp = \fk\cap\fp$.

The unit sphere bundle $US^n$ is also a contact manifold, and this structure is one of its most important features for
our purposes. Recall that for any Riemannian manifold $(N,g)$ the tangent bundle carries a canonical $1$-form $\theta$,
which can be defined by
\[
\theta(Z) = g(d\pi(Z),\pi_{TN}(Z)),
\]
where $\pi_{TN}:T(TN)\to TN$ is the double tangent bundle's projection.  It is well-known (see, for example \cite[Ex.\
3.44]{McDS}) that when $\theta$ is restricted to $UN$ it becomes a contact form, and the corresponding Reeb vector field
$\caZ$ generates geodesic flow along $UN$. When $N=S^n$ (or indeed any CROSS) geodesic flow provides a circle action,
since every geodesic has the same length. In the homogeneous geometry this circle action is the right action of the
circle subgroup $S = \{\exp(t\nu_0):t\in\R\}$ on $G/H$, which is well-defined since the right actions of $H$ and $S$
commute. Thus $US^n$ has a well-defined quotient by geodesic flow, which
we will denote by $\caQ$, and which is isomorphic to the Grassmanian of oriented $2$-planes in $\R^{n+1}$:
\[
\caQ = US^n/S \simeq G/(H\times S)\simeq SO(n+1)/(SO(n-1)\times SO(2))\simeq \Gr(2,n+1).
\]
We can think of $\caQ$ as the manifold which parameterises oriented geodesics for $S^n$. 
If we set $L=H\times S\subset G$ then the reductive decomposition for $G/L$ is
\[
\fg = \fl + \fq,\quad \fl = \fh+\fs,\quad \fq  =\fl^\perp=\fp\cap\fs^\perp,
\]
where $\fs = \R.\nu_0$. Thus $T\caQ\simeq [\fq]_L$ and when we equip $\caQ$ with the normal metric $h_\caQ$ compatible with
$\g{\ }{\ }$ the quotient map $\pi_\caQ:US^n\to \caQ $ is a Riemannian submersion with
geodesic fibres.  With this metric $\caQ\simeq \Gr(2,n+1)$ is a symmetric space, with
$[\fq,\fq]\subset \fl$, and this implies the following bracket relations which we will make use of later:
\begin{equation}\label{eq:brackets}
[\fn,\fn]\subset\fh,\quad [\fn,\fm_0]\subset \fs,\quad [\fm_0,\fm_0]\subset \fh.
\end{equation} 
Here $\fm_0 = \fm\cap\fq = \fm\cap\fs^\perp$, and we have also used the symmetric space relations for $\fg=\fk+\fm$.
\begin{rem}\label{rem:spheres}
Of course, the first bracket is implicit in the statement that $S^{n-1}\simeq K/H$ is a symmetric space.
The last bracket, together with $[\fh,\fm_0]\subset\fm_0$ from $[\fk,\fm]\subset\fm$, implies that $\bar\fk = \fh+\fm_0$
is also a Lie subalgebra. In fact $\fk\simeq\bar \fk$, so for the Lie subgroup $\bar K$ we also have $S^n\simeq G/\bar K$
and $S^{n-1}\simeq\bar K/H$. This gives us two homogeneous projections
\[
\pi:G/H\to G/K,\quad \bar\pi:G/H\to G/\bar K.
\]
The vertical distribution for $\pi$ is $\caV$, while that for $\bar\pi$ is $\caH_0=[\fm_0]_H=\caH\cap\caC$.
Thus each of these distributions is integrable, and each gives a foliation by totally geodesic $(n-1)$-spheres. 
The fibres of $\bar\pi$ are just the horizontal lifts of totally geodesic $(n-1)$-spheres in $G/K$.
Since each of these foliations is $\pi_\caQ$-horizontal they descend to $\caQ$, where their leaves are totally geodesic
Lagrangian $(n-1)$-spheres.  If we view $US^n$ as $\{(p,v)\in S^n\times S^n:p\cdot v=0\}$ these fibrations are 
\[
\pi,\bar\pi:US^n\to S^n,\quad \pi(p,v)=p,\quad \bar\pi(p,v) = v.
\]
They map into each other under the automorphism $(p,v)\mapsto (v,-p)$, which corresponds to
rotation through $\pi/2$ under the geodesic flow. 
\end{rem}

Whenever it is possible, taking the quotient by geodesic flow is an example of symplectic reduction: 
geodesic flow is the Hamiltonian 
flow corresponding to the squared-length function $\ell:TN\to \R$, $\ell(X) = |X|^2$, and $UN/S = \ell^{-1}(1)/S$
carries the symplectic $2$-form $\lambda_\caQ$ which descends from $\lambda = -d\theta$. For $S^n$ 
we will show below that $\theta = h(\caZ,\cdot)$ and therefore the contact distribution $\caC = \ker\theta$ is also the
horizontal distribution $\caZ^\perp$ for $\pi_\caQ$. Moreover the metric and symplectic structures which descend from $US^n$
are compatible and give the unique (up to scale) $G$-invariant \Kah-Einstein structure on $\Gr(2,n+1)$.

The following lemma summarizes what we need to know about the projection $\pi_\caQ:US^n\to\caQ$ using homogeneous geometry.
\begin{lem}
On $US^n\simeq G/H$, the canonical $1$-form $\theta$ is the $G$-invariant $1$-form corresponding to the
$\Ad_H$-invariant linear form $\theta_0:\fp\to\R$ defined by $\theta_0(\xi) = \g{\nu_0}{\xi}$. Consequently, 
$\caC\simeq [\fq]_H\subset [\fp]_H$. The $Ad_L$-invariant linear map
\[
J_0:\fp\to\fp;\quad \eta\mapsto [\nu_0,\eta],
\]
corresponds to an endomorphism $J \in\End(TUS^n)$ with $J\caZ=0$ and which induces an almost complex structure in
$\caC$. Finally, $\lambda=-d\theta$ is the $G$-invariant $2$-form corresponding to the $\Ad_L$-invariant skew bilinear form
$\lambda_0:\fp\times\fp\to\R$ defined by 
\begin{equation}\label{eq:lambda0}
\lambda_0(\xi,\eta) = \g{\nu_0}{[\xi,\eta]}=\g{J_0\xi}{\eta}. 
\end{equation}
The pair $(\lambda,J)$ therefore both descend to $\caQ\simeq G/L$, where they give the unique $G$-invariant 
KE structure $(\lambda_\caQ,J_\caQ)$ (i.e., the Hermitian symmetric space structure) compatible with
the normal metric $h_\caQ$.
\end{lem}
Our main point here is to make explicit the relationship between the structures $\theta,J,\lambda$ on $US^n$ and
$J_\caQ,\lambda_\caQ$ on $\caQ$.
\begin{proof}
First, $\theta$ is clearly $G$-invariant since the metric $g$ is and the projections $\pi$ and $\pi_{TN}$ are
$G$-equivariant. Now about any $X\in US^n$ we may choose a local frame $\Phi$ of $G\to G/H$, so that
$\beta_K(X) = \Ad\Phi\cdot\nu_0$. For $Y \in T_XUS^n$ we have
\begin{align*}
\theta_X(Y) &= \g{\beta_K(X)}{\beta_K(d\pi(Y))}\\
& = \g{\Ad\Phi\cdot\nu_0}{\Ad \Phi\cdot\phi_\fm(Y)}\\
& = \g{\nu_0}{\phi_\fp}
\end{align*}
for $\phi = \Phi^{-1}d\Phi$, using the fact that for $\xi\in\fm$ and $\eta\in\fp$, $\g{\xi}{\eta_\fm}= \g{\xi}{\eta}$. 
This yields the expression for $\theta_0$ above. To obtain the expression for $\lambda$ we note that
$\theta = \g{\nu_0}{\phi}$, since $[\nu_0,\fh]=0$, and therefore
\[
d\theta = \g{\nu_0}{d\phi} = \g{\nu_0}{-\tfrac12[\phi\wedge\phi]},
\]
by the Maurer-Cartan equations. Thus, writing $\eta=\phi(Y)$ and $\zeta=\phi(Z)$, we have
\begin{align*}
\lambda(Y,Z) &= \g{\nu_0}{[\eta,\zeta]}\\
&=\g{\nu_0}{[\eta_\fh,\zeta_\fh]+[\eta_\fh,\zeta_\fp]+ [\eta_\fp,\zeta_\fh] + [\eta_\fp,\zeta_\fp]}\\
&= \g{\nu_0}{[\eta_\fp,\zeta_\fp]}
\end{align*}
where we have used the adjoint invariance of the inner product to eliminate terms. One can check directly
that $\ad^2\nu_0=-I$ on $\fq$, but the geometric reason is that
$S^n$ has constant curvature $1$, since the curvature on any symmetric
space $G/K$ comes from the $\Ad_K$-invariant tri-linear map
\begin{equation}\label{eq:R}
R_0:\fm\times\fm\times\fm \to\fm,\quad R_0(\xi,\eta)\zeta = -[[\xi,\eta],\zeta].
\end{equation}
In constant curvature $1$, this gives, for $\eta\in\fm_0 = \fm\cap\fs^\perp$,
\[
-\ad^2\nu_0\cdot\eta = R_0(\eta,\nu_0)\nu_0 = \g{\nu_0}{\nu_0}\eta - \g{\eta}{\nu_0}\nu_0 = \eta
\]
Now $\ad\nu_0:\fm_0\to\fn$ is an isomorphism, since it is injective and $\dim\fm_0=n-1=\dim\fn$, so the same
conclusion holds on $\fn$. As a consequence, if we set $\caH_0=\caH\cap\caC$, then
\begin{equation}\label{eq:J}
J:\caV \to \caH_0,\quad J:\caH_0\to\caV,
\end{equation}
are both isomorphisms. The rest of the lemma follows since $(h_\caQ,J_\caQ)$ is plainly the standard Hermitian 
symmetric space structure on 
$G/L\simeq \Gr(2,n+1)$ corresponding to the centre $S$ of its isotropy subgroup $L$. 
\end{proof}
We finish this section with some facts about $US^n$ which we will need in the next section. The first describes 
the geometric meaning of the vertical projection $TUS^n\to\caV$ along $\caH$.
\begin{lem}\label{lem:vertZ}
Let $Z\in T_\xi US^n$ be non-vertical, with vertical component $Z^\caV$. Then 
\begin{equation}\label{eq:vertZ}
d\pi(JZ^\caV) = -\nabla_{\bar Z} Y
\end{equation} 
where $\bar Z = d\pi(Z)$ and $Y(t)$ is any curve in $US^n$ satisfying $Y(0)=\xi$, $\dot{Y}(0) = Z$. 
Hence $Z$ is horizontal whenever $Y(t)$ is parallel along $\pi(Y(t))$.
\end{lem} 
\begin{proof}
Let $Y(t)$ be any curve in $G/H$ satisfying $Y(0)=\xi$, $\dot{Y}(0) = Z$ and frame it by a curve $g(t)$ in $G$, i.e, 
$Y(t) = g(t)H$ and therefore $\beta_K(Y)= \Ad g\cdot \nu_0$. Set $\eta = (g^{-1}\dot{g})(0)$, so that
\[
\beta_H(Z)=\beta_H(\dot{Y}(0)) = \Ad g(0)\cdot\eta_\fp,
\]
and therefore $\beta_H(Z^\caV) = \Ad g(0)\cdot\eta_\fn$. On the other hand 
$\bar Z = (d\pi\circ Y/dt)(0)$ and therefore
\begin{align*}
\beta_K(\nabla_{\bar Z} Y)& = \Ad g(0)\cdot\{\frac{d}{dt}\nu_0 + [\eta,\nu_0]\}_\fm\\
&= \Ad g(0)\cdot [\eta_\fn,\nu_0]\\
& = -\beta_K(d\pi(JZ^\caV)),
\end{align*}
since $[\eta,\nu_0]_\fm = [\eta_\fn,\nu_0]$ by \eqref{eq:brackets}.
\end{proof}
It follows from this that the metric $h$ agrees with the Sasaki metric $h_s$, since, in the notation of the previous
lemma, 
\[
h_s(Z,Z) = g(Z^\caH,Z^\caH) + g(\nabla_{\bar Z}Y,\nabla_{\bar Z}Y) = h(Z^\caH,Z^\caH) + h(JZ^\caV,JZ^\caV).
\]
and $J$ is an isometry within $\caC\supset\caV$. 

In the previous proof we used the expression for the Levi-Civita connexion 
$\nabla$ of $(G/K,g)$ in a local frame. Similarly, we note for later use that in a local frame $\Phi:V\to G$ for $G/H$ 
the Levi-Civita connexion $\nabla^h$ for $(US^n,h)$ takes the form
\begin{equation}\label{eq:LC}
\beta_H(\nabla^h_ZW)  = \Ad\Phi\cdot\{Z\phi_\fp(W) + [\phi_\fh(Z),\phi_\fp(W)]+[\phi_\fp(Z),\phi_\fp(W)]_\fp\},
\end{equation}
for $Z,W\in\Gamma(TUS^n)$ and $\phi = \Phi^{-1}d\Phi$.

\section{The spherical and geodesic Gauss maps.}

Let $f:M\to S^n$ be a smooth immersion of an $m$-dimensional oriented manifold $M$, which we equip with the induced metric. 
The normal bundle $TM^\perp$ is the orthogonal complement to $TM\subset f^{-1}TS^n$ and we will write
$\pi^\perp:UM^\perp\to M$ for the unit normal bundle. Then we have an induced immersion $\mu:UM^\perp\to US^n$. Following
\cite{JenR} we will call this the \emph{spherical Gauss map} (not to be confused with Obata's Gauss map \cite{Oba}, which is
sometimes referred to by this name). This is
a Legendrian immersion for the contact structure $\caC$ on $US^n$ \cite[Ex.\ 3.44]{McDS}, and therefore it is horizontal for the
projection $\pi_\caQ:US^n\to\caQ$. The composite $\gamma:UM^\perp\to\caQ$ assigns to each normal vector the oriented geodesic
it generates, and therefore we will call this the \emph{geodesic Gauss map}. This is a Lagrangian immersion, since $\mu$ is
Legendrian. Since $\pi_\caQ$ is a Riemannian submersion, the metric on $UM^\perp$ induced by $\mu$ agrees with that induced by
$\gamma$, and this will be the metric we equip $UM^\perp$ with. The following diagram summarizes the relation between maps
(cf.\ \cite{MaO11}, where $\caQ$ is treated as the complex hyperquadric):
\begin{equation}\label{eq:gammaf}
\begin{array}{ccccc}
UM^\perp & \stackrel{\mu}{\longrightarrow} & US^n & & \\
\pi^\perp\downarrow\quad  & &\quad \downarrow\pi&\stackrel{\pi_\caQ}{\searrow} & \\
M & \stackrel{f}{\longrightarrow} & S^n & &\caQ
\end{array}
\end{equation} 

We will denote the induced connexion, shape operator and second fundamental form of $f$ by $\nabla^f$, $A_f$ and $\II_f$ 
(and similarly for $\mu$ and
$\gamma$), and take its mean curvature to be $H_f=\tr_g\II_f$. Our aim is to describe the relationship between the mean
curvatures $H_f$, $H_\mu$ and $H_\gamma$. For the last two the relationship is straighforward and follows from the next
lemma regarding the behaviour of tension fields.
\begin{lem}\label{lem:tension}
Let $U,P,Q$ be Riemannian manifolds.
Suppose $\varphi:U\to P$ is a smooth immersion and $\psi:P\to Q$ is a Riemannian submersion.
If $\varphi$ is horizontal then its tension field $\tau(\varphi)$ is also horizontal and
$\tau(\psi\circ\varphi)=\psi_*\tau(\varphi)$. Hence $\varphi$ is harmonic if and only if
$\psi\circ\varphi$ is harmonic. Further, if $\varphi$ is isometric then so is $\psi\circ\varphi$, 
therefore $H_{\psi\circ\varphi} = \psi_*H_\varphi$, and $\varphi$ is minimal if and only if $\psi\circ\varphi$ is minimal.
\end{lem}
\begin{proof}
Set $\gamma = \psi\circ\varphi$, then we can think of $\varphi$ as its horizontal lift. Notice
that $\gamma$ is an immersion since $\varphi$ is a horizontal immersion.
First we show that $\psi_*\tau(\varphi) = \tau(\gamma)$.
Let $E_i$ be a locally orthonormal frame for $U$. Then (using summation convention)
\begin{eqnarray*}
\psi_*\tau(\varphi) & = & \psi_*(\nabla^P_{d\varphi(E_j)}\varphi_*E_j -
d\varphi(\nabla^U_{E_j}E_j))\\
&=& \nabla^Q_{d\gamma(E_j)}\gamma_*E_j - d\gamma(\nabla^U_{E_j}E_j)\\
&=& \tau(\gamma). 
\end{eqnarray*}
Here we used \cite[Lemmas 2,3]{ONe} to deduce that for $X,Y\in\Gamma(TU)$, since $\varphi_*Y$ is
the horizontal lift of $\gamma_*Y$,
\[
\psi_*(\nabla^P_{d\varphi(X)}\varphi_*Y)=\nabla^Q_{d\gamma(X)}\gamma_*Y.
\]
Now the vertical component of $\tau(\varphi)$ is the vertical component of
$\nabla^P_{d\varphi(E_j)}\varphi_*E_j$.
But by \cite[Lemma 3]{ONe} $\nabla^P_XX$ is horizontal whenever $X$ is a horizontal lift.
Hence $\tau(\varphi)$ is horizontal. The final statement is obvious since $\psi$ is a Riemannian
submersion.
\end{proof}
Since $\mu$ and $\gamma$ are isometric immersions we have $H_\mu=\tau(\mu)$ and $H_\gamma=\tau(\gamma)$. The conditions of
the lemma apply to our situation, and therefore $H_\mu$ lies in $\caC$ and $H_\gamma= (\pi_\caQ)_*H_\mu$.

To be able to compare $H_\mu$ with $H_f$ we need to understand how the bundle map $\pi^\perp:UM^\perp\to M$ relates the
induced metrics $\mu^*h$ and $f^*g$. Some of what we want to
know has already been established by Gudmundsson \& Mo \cite{GudM} in their study of unit normal bundles as harmonic
morphisms. First we need a preliminary observation. Let $T(UM^\perp) = \caH_M+\caV_M$ denote the decomposition into
$\pi^\perp$-horizontal and vertical subbundles. Although $\caV_M\subset \mu^{-1}\caV$, it is almost never true that $\caH_M\subset
\mu^{-1}\caH$. In fact this is only true when $f$ is a totally geodesic immersion, for the following reason.
\begin{lem}\label{lem:vertJZ}
Let $Z\in T_\xi(UM^\perp)$ be $\pi^\perp$-horizontal, and write it as $Z = Z^\caH+Z^\caV$ for the splitting
$\mu^{-1}\caH+\mu^{-1}\caV$ and notice that $Z^\caH\in\caH\cap\caC$. Then 
\[
d\pi(JZ^\caV) = A_f(\xi)\bar Z,
\]
where $\bar Z=d\pi(\caZ^\caH)$ we think of $J$ as an almost complex structure in $\caC=[\fq]_H$.
Hence $\pi^\perp$ is a Riemannian submersion if and only if $A_f=0$, i.e., $f$ is
totally geodesic.
\end{lem}
For the proof we need to set up some frame conventions which we will use thoughout the rest of the section. About each
$\xi\in UM^\perp$ we can choose a local frame $\Phi:V\to G$ for $\mu$, and provided we choose $V$ carefully, so that it is
both simply connected and its image $U=\pi(V)\subset M$ is open, there will be a local frame $F:U\to G$ for $f$ related
by $\Phi = (F\circ\pi)\Psi$ for some $\Psi:V\to K$. Set $\phi = \Phi^{-1}d\Phi$, $\alpha = F^{-1}dF$. Then
\[
\phi = \Ad\Psi^{-1}\cdot(\alpha\circ d\pi) + \psi,\quad \psi = \Psi^{-1}d\Psi.
\]
Since $\psi$ takes values in $\fk$, this means
\begin{equation}\label{eq:phim}
\phi_\fk = \Ad\Psi^{-1}\cdot(\alpha\circ d\pi)_\fk + \psi,\quad \phi_\fm = \Ad\Psi^{-1}\cdot(\alpha\circ d\pi)_\fm.
\end{equation}
In particular, $\Ad\Phi\cdot\phi_\fm(Z)= \Ad F\cdot\alpha_\fm(d\pi(Z))$. We also note that whenever $Z\in\caC$
\begin{equation}\label{eq:Jphi}
\phi_\fp(JZ) = J_0\phi_\fp(Z)\implies \phi_\fn(JZ) = J_0\phi_\fm(Z),\ \phi_\fm(JZ) = J_0\phi_\fn(Z).
\end{equation}
\begin{proof}
Using Lemma \ref{lem:vertZ}, we have
\[
\phi_\fm(JZ^\caV) = J_0\phi_\fn(Z^\caV) = [\nu_0,\phi_\fn(Z^\caV)] = -\phi_\fm(\nabla_{\bar Z}\xi), 
\]
where $\xi$ follows any curve tangent to $Z$. Now since $Z\in(\caH_M)_\xi$ we have, for all $W\in(\caV_M)_\xi$,
\begin{align*}
0= h(Z,W) & = h(Z^\caV,W) \\
& = \g{J_0\phi_\fn(Z)}{J_0\phi_\fn(W)}\\
&= \g{\phi_\fm(\nabla_{\bar Z}\xi)}{\phi_\fm(JW)}\\
& = g(\nabla_{\bar Z}\xi,\bar W).
\end{align*}
Now for every $\xi\in UM^\perp$, $TM^\perp = d\pi(J(\caV_M)_\xi)+\R.\xi$, and
$g(\nabla_{\bar Z}\xi,\xi)=0$. Therefore $(\nabla_{\bar Z}\xi)^\perp=0$.  Hence $\nabla_{\bar Z}\xi = -A_f(\xi)\bar Z$. 
\end{proof}
Therefore we can write every $\pi^\perp$-horizontal vector $Z\in T_\xi(UM^\perp)$ as
\[
Z = \bar Z - JA_f(\xi)\bar Z,
\]
if we identify $\bar Z=d\pi(Z)$ with its $\pi$-horizontal lift to $T_\xi US^n$. Consequently in $\caH_M$ the induced metric 
has the form
\begin{equation}\label{eq:hHM}
h(Z,W) = g(\bar Z,\bar W) + g(A_f(\xi)\bar Z,A_f(\xi)\bar W),\quad Z,W\in\caH_M.
\end{equation}
Indeed, this holds for any unit normal bundle equipped with the metric induced by the Sasaki metric \cite{GudM}.
This prompts the following definition.
\begin{defn}
For an isometric immersion $f:M\to N$ into a Riemannian manifold $(N,g)$, we will define its \emph{shape form} 
to be the operator
\[
a_f:TM^\perp\to S^2TM^*,\quad  a_f(\xi)(X,Y) = g(A_f(\xi)X,A_f(\xi)Y),
\] 
for $\xi\in T_pM^\perp$, $X,Y\in T_pM$. 
We will say $f$ has \emph{(weakly) conformal shape form} when there exists a smooth function $r:UM^\perp\to\R$ for which
$a_f(\xi) = r(\xi)^2 g$ for all $\xi\in UM^\perp$. This is equivalent to the condition that 
$A_f(\xi)^2 = r(\xi)^2I$ for all $\xi\in UM^\perp$, and therefore $r(\xi)^2 = \tfrac{1}{m}\|A_f(\xi)\|^2$.
\end{defn}
Notice that $a_f(\xi)$ is positive semi-definite for all $\xi\neq 0$, since $A_f$ is symmetric.

Gudmundsson \& Mo \cite{GudM,Mo} use the terminology \emph{conformal second fundamental form} to refer to the condition
defined above. But that phrase was used much earlier by Jensen \& Rigoli \cite{JenR} to refer to a very different condition, 
and so we have chosen to avoid confusion by providing a different name.

Given the expression \eqref{eq:hHM} we deduce immediately:
\begin{lem}[\cite{GudM} Prop.\ 4.2]\label{lem:horizconf}
The submersion $\pi^\perp:UM^\perp\to M$ is horizontally conformal if and only if 
$f$ has conformal shape form. The conformal factor is $1/(1+r(\xi)^2)$. In particular, every minimal surface has conformal
shape form.
\end{lem}
The last statement is easy to see: $f$ is minimal if and only if  $\tr A_f(\xi)=0$ for every unit normal vector $\xi$. 
In dimension two this means $A_f(\xi)^2 = -\det(A_f(\xi)) I$ and $\det(A_f(\xi))<0$. 
We can now state our main result.
\begin{thm}\label{thm:main}
Suppose $f:M\to S^n$ has conformal shape form. Then at $\xi\in UM^\perp$, for $Z,W\in T_\xi(UM^\perp)$ with $Z$
$\pi^\perp$-horizontal and $W$ vertical, 
\begin{align}
\lambda(H_\mu,Z) & = \tfrac{1}{1+r(\xi)^2} g(\nabla^\perp_{\bar Z}H_f,\xi),\label{eq:HZ}\\
\lambda(H_\mu,W) & = - \tfrac{1}{1+r(\xi)^2} g(H_f,\bar{W}),\label{eq:HW}
\end{align}
where $r(\xi)^2 = \tfrac{1}{m}\tr_gA_f(\xi)^2$, $\bar{Z}=d\pi(Z)$ and $\bar{W}=d\pi(JW)$. 
For hypersurfaces the second equation is vacuous. 
Thus for $m<n-1$, $\mu$ is minimal (equally, $\gamma$ is minimal
Lagrangian) if and only if $f$ is minimal. For hypersurfaces $\mu$ is minimal if and only if $f$ has constant mean
curvature.
\end{thm}
We can be a little more precise about what happens in dimension $m\geq 3$ given the condition of conformal shape form. 
For $f$ to be minimal and have conformal shape form it must have both $\tr A_f(\xi)=0$ and
$A_f(\xi)^2 = r(\xi)^2I$ for all $\xi$. By considering the eigenvalues of $A_f(\xi)$ it is easy to see that
when $m$ is odd this forces $r(\xi)=0$, so $f$ must be be totally geodesic. This, and some other implications of conformal
shape form, can be found \cite{Mo}.  It is also quite restrictive when $M$ is a constant mean curvature
hypersurface for $m\geq 3$: we discuss this in Example \ref{exam:hyper} below.

Before we give the proof of this theorem, we consider some examples and how these fit in with the literature.
\begin{exam}\label{exam:totgeod}
The simplest examples arise when $M$ is a totally geodesic $m$-sphere. As we have mentioned, this is the
only case for which $\pi^\perp$ is a Riemannian submersion. We can represent $M$ as $V\cap S^n$, where
$V\subset\R^{n+1}$ is a linear subspace of dimension $m+1$, and then it is easy to see that 
\[
UM^\perp = (V\times V^\perp)\cap (S^n\times S^n) \subset US^n\subset S^n\times S^n.
\]
So $UM^\perp\simeq S^m\times S^{n-m-1}$. For each $m$ these are isometric to each other, and are homogeneous
submanifolds for a subgroup of $G$ isomorphic to $SO(m+1)\times SO(n-m)$. These are totally geodesic in both $US^n$
and $\caQ$. In particular, when $m=1$ and $n\geq 3$ these give the only examples, since $M$ must be a geodesic. 
The two foliations of $\caQ$ described in Remark \ref{rem:spheres} fit into this class of examples: they represent the 
cases of $m=n-1$ and $m=0$.
\end{exam}
\begin{exam}\label{exam:surfaces}
By Lemma \ref{lem:horizconf} and the theorem, for every minimal surface in $S^n$ its unit normal bundle is minimal 
Legendrian in $US^n$. By projection into $\caQ$ we obtain minimal Lagrangian submanifolds each  
of which is an $S^{n-3}$-bundle over a surface, and where each fibre is totally geodesic. 
This provides a vast supply of compact minimal Lagrangian submanifolds in $\caQ$, since there is a
very rich theory of compact minimal surfaces in spheres: in $S^3$ from Lawson's examples \cite{Law} and 
Hitchin's construction of all minimal tori \cite{Hit}; in $S^4$, from Bryant's superminimal surfaces \cite{Bry} and the
non-superminimal tori of \cite{FerPPS}; in higher dimensions from Riemannian twistor theory 
and integrable systems constructions. For surfaces in $S^3$ the geodesic Gauss map agrees with the usual Gauss
map (for a surface in $\R^4$). Palmer \cite{Pal94} seems to have been the first to point out that minimal
surfaces in $S^3$ have minimal Lagrangian Gauss maps; see also \cite{CasU}.
\end{exam}
\begin{exam}\label{exam:hyper}
When $M$ is a hypersurface we may as well restrict our attention to one connected component of $UM^\perp$ and identify
this with $M$. The geodesic Gauss map is then just the usual one, considering $M\subset\R^{n+1}$. Palmer
\cite{Pal97} gave a formula for $H_\gamma$ which applies in all cases, not just with the assumption of conformal
shape form (see Remark \ref{rem:hypersurf} below). From this formula he deduced that the Gauss map would be minimal
Lagrangian whenever $M$ is isoparametric, i.e., has constant principal curvatures. Palmer's formula has been the basis for
subsequent constructions of minimal Lagrangian submanifolds: from isoparametric hypersurfaces
\cite{MaO09,Ohn} and from rotationally symmetric (non-isoparametric) hypersurfaces \cite{LiMW}. The previous theorem 
yields nothing beyond this for hypersurfaces of dimension $3$ or more, because the condition that $f$ be CMC and have 
conformal shape form forces
$f$ to be isoparametric. To see why, we compare $H_\gamma$, which is the tension field for $\gamma:(M,\mu^*h)\to\caQ$, 
with the tension field $\tau(\bar\gamma)$ for $\bar\gamma:(M,f^*g)\to \caQ$ (the same map with a different metric on $M$). 
We can write
$\gamma = \bar\gamma\circ\pi^\perp$, and then the
composition formula for tension fields \cite[2.20]{EelL} says that
\[
H_\gamma = d\bar\gamma(\tau(\pi^\perp)) + \tr_{h_\caQ}\nabla d\bar\gamma(d\pi^\perp,d\pi^\perp).
\]
When $f$ has conformal shape form, with $A_f^2(\xi) = r(\xi)^2I$, the last term is 
$\tfrac{1}{1+r^2}\tau(\bar\gamma)$. When $f$ is additionally CMC, $H_\gamma=0$ by the previous theorem and
$\tau(\bar\gamma)=0$ by the Ruh-Vilms theorem \cite{RuhV}. Therefore $\pi^\perp$ is conformal and harmonic, and therefore a
harmonic morphism (with zero dimensional fibres). When $m\geq 3$ this
implies that $\pi^\perp$ is a homothety \cite[Thm 2]{Gud}, so $r(\xi)$ is constant. 
\end{exam}
To prove Theorem \ref{thm:main} we establish the general expressions for $H_\mu$ without the assumption of conformal shape form. 
We first set up an adapted orthonormal frame about each
point $\xi\in UM^\perp$. Since $UM^\perp$ is Legendrian, about each point $\xi\in UM^\perp$ 
we can choose a frame for $UM^\perp$ by taking local orthonormal
frames $\{E_i:I=1,\ldots,m\}$ for $\caH_M$ and $\{E_\beta:\beta = m+1,\ldots,n-1\}$ for $\caV_M$. Then
\[
\{E_i,E_\beta,JE_i,JE_\beta: 1\leq i\leq m,\ m+1\leq \beta \leq n-1\},
\]
frames $\mu^{-1}\caC$ orthonormally about $\xi$. 
From now on we will use the index conventions $1\leq A,B\leq n-1$, $1\leq i,j\leq m$ and $m+1\leq\beta\leq n-1$.
As before we define $\bar E_i = d\pi(E_i)$ and $\bar E_\beta = d\pi(JE_\beta)$ (note that $JE_\beta\in\caH_\xi$).
Then $TM = \Sp\{\bar E_i\}$ and $TM^\perp = \Sp\{\bar E_\beta,\xi\}$.
\begin{prop}
At $\xi\in UM^\perp$, with the adapted frame above and $p=\pi(\xi)$,
\begin{align}
h_\xi(H_\mu, JE_i) & =  -g_p(\sum_j[\nabla\II_f](\bar E_i,\bar E_j,\bar E_j),\xi),\label{eq:Hi}\\
h_\xi(H_\mu, JE_\beta) & = g_p(\sum_j \II_f(\bar E_j,\bar E_j),\bar E_\beta).\label{eq:Hbeta}
\end{align}
\end{prop}
Theorem \ref{thm:main} is now a direct consequence of these expressions in the case when $f$ has conformal shape form,
since $\|\bar E_i\|^2 = 1/(1+r(\xi)^2)$ at $\xi$ in that case (and of course 
$\tr_g\nabla\II_f = \nabla^\perp \tr_g\II_f$).
\begin{proof}
By Lemma \ref{lem:tension} the mean curvature $H_\mu$ lies in $\mu^{-1}\caC$ so we need only calculate 
\[
h(H_\mu,JE_B)= \sum_Ah(\II_\mu(E_A,E_A),JE_B) = \sum_{A} h(\nabla^h_{E_A}\mu_*E_A,JE_B). 
\]
For simplicity, set $H_\mu^i = h(H_\mu,JE_i)$ and $H_\mu^\beta = h(H_\mu,JE_\beta)$. 
Using the expression \eqref{eq:LC} we have
\[
\beta_H(\nabla^\mu_{E_A}E_A) = \Ad\Phi\cdot (E_A\phi_\fp(E_A) + [\phi_\fh(E_A),\phi_\fp(E_A)]).
\]
Since each $E_\beta$ is vertical it has $\phi_\fm(E_\beta)=0=\phi_\fn(JE_\beta)$, so that
\begin{align}\label{eq:Hi1}
H_\mu^i & = \sum_A\g{E_A\phi_\fp(E_A) + [\phi_\fh(E_A),\phi_\fp(E_A)]}{\phi_\fp(JE_i)}\notag\\
&= \sum_j \g{E_j\phi_\fp(E_j) + [\phi_\fh(E_j),\phi_\fp(E_j)]}{\phi_\fp(JE_i)}\notag\\
&+ \sum_\beta \g{E_\beta\phi_\fn(E_\beta) + [\phi_\fh(E_\beta),\phi_\fn(E_\beta)]}{\phi_\fn(JE_i)},
\end{align}
and
\begin{equation}\label{eq:Hbeta1}
H_\mu^\beta  = \g{\sum_jE_j\phi_\fm(E_j) + [\phi_\fh(E_j),\phi_\fm(E_j)]}{\phi_\fm(JE_\beta)}.
\end{equation}
The second summand in \eqref{eq:Hi1} vanishes because the fibre $U_pM^\perp$ is totally geodesic. 
Each $E_\beta$ is tangent to this fibre, and the Levi-Civita connexion along this fibre is the restriction of
\eqref{eq:LC} to the case where $\phi_\fm(Z)=0=\phi_\fm(W)$. 

Therefore
\begin{align}\label{eq:Hi2}
H_\mu^i & = \sum_j \g{E_j\phi_\fm(E_j) + [\phi_\fh(E_j),\phi_\fm(E_j)]}{\phi_\fm(JE_i)}\notag\\
&+ \sum_j \g{E_j\phi_\fn(E_j) + [\phi_\fh(E_j),\phi_\fn(E_j)]}{\phi_\fn(JE_i)}.
\end{align}
To simplify this define, for $X,Y$ local vector fields on $UM^\perp$, 
\[
D_X\phi_\fp(Y) = X\phi_\fp(Y) + [\phi_\fh(X),\phi_\fp(Y)].
\]
This represents the (pullback to $UM^\perp$ of the) canonical connexion of $G/H$ in the local frame $\Phi$. We recall
that $\phi_\fp(JX) = J_0\phi_\fp(X)$ and observe that $D_X\phi_p(JX) = J_0D_X\phi_p(X)$ since $J_0$ is $\Ad_H$-invariant. 
We claim that whenever $X\in(\caH_M)_\xi$ and $Y,Z$ are local sections of $\caC$ about $\xi$, we have
\begin{equation}\label{eq:D_X}
\g{D_X\phi_\fm(Y)}{\phi_\fm(Z)} = g(\nabla_{\bar X}\bar Y,\bar Z).
\end{equation}
To see this, we use $\phi_\fh = \phi_\fk-\phi_\fn$ to write the left hand side as
\[
 \g{X\phi_\fm(Y) + [\phi_\fk(X),\phi_\fm(Y)]}{\phi_\fm(Z)} - \g{[\phi_\fn(X),\phi_\fm(Y)]}{\phi_\fm(Z)}.
\]
Now we observe that, from Lemma \ref{lem:vertJZ}, since $X\in\caH_M$,
\begin{equation}\label{eq:phiX}
\phi_\fn(X) = -J_0\phi_\fm(A_f(\xi)\bar X),\quad 
\phi_\fm(JX)  = J_0\phi_\fn(X) = \phi_\fm(A_f(\xi)\bar X),\quad 
\end{equation}
and $\phi_\fn(JX)  = J_0\phi_\fm(\bar X)$. We also note that \eqref{eq:phim} implies 
\begin{align*}
& \Ad\Phi\cdot(X\phi_\fm(Y) +[\phi_\fk(X),\phi_\fm(Y)])\\
 = & \Ad(F\circ\pi_N)\cdot
(\sum_{j=1}^m \bar X\alpha_\fm(\bar Y) +[\alpha_\fk(\bar X),\alpha_\fm(\bar Y)]).
\end{align*}
Therefore the left hand side of \eqref{eq:D_X} becomes
\begin{align*}
& \g{\bar X\alpha_\fm(\bar Y) + [\alpha_\fk(\bar X),\alpha_\fm(\bar Y)]}{\alpha_\fm(\bar Z)}
+ \g{[[\nu_0,\phi_\fm(A_f(\xi)\bar X)],\phi_\fm(Y)]}{\phi_\fm(Z)},\\
 & = g(\nabla_{\bar X}\bar Y,\bar Z) - g(R(\xi,A_f(\xi)\bar X)\bar Y,\bar Z),
\end{align*}
using the curvature expression \eqref{eq:R} and \eqref{eq:phim}, together with the fact that
at the point $\xi$ itself, $\beta_K(\xi) = \Ad\Phi\cdot\nu_0$ since $\Phi$ is a frame for $\mu$. 
But in $S^n$ this curvature term equals
\[
g(\xi,\bar Z)g(A(\xi)\bar X,\bar Y)-g(\xi,\bar Y)g(A(\xi)\bar X,\bar Z)=0,
\]
since $d\pi(\caC_\xi)\perp\xi$. This establishes \eqref{eq:D_X}. 

Now using \eqref{eq:D_X} and \eqref{eq:phiX} we compute the terms in \eqref{eq:Hi2}:
\begin{align*}
\g{ D_{E_j}\phi_\fm(E_j)}{\phi_\fm(JE_i)} & = g(\nabla_{\bar E_j}\bar E_j,A_f(\xi)\bar E_i)
  = g(A_f(\xi)\nabla^f_{\bar E_j}\bar E_j,\bar E_i),\\
\g{ D_{E_j}\phi_\fn(E_j)}{\phi_\fn(JE_i)} & = -\g{D_{E_j}\phi_\fm(JE_j)}{\phi_\fm(E_i)}
= -g(\nabla^f_{\bar E_j}[A_f(\xi)\bar E_j],\bar E_i),
\end{align*}
where we have also used that fact that $A_f(\xi)$ is self-adjoint and $J_0$ is an isometry. In this second term $\xi$
moves along the integral curve of $E_j$. Now $(\nabla_{\bar E_j}\xi)^\perp=0$, so the sum of the two expressions
can be simplified as follows:
\begin{align*}
g(A_f(\xi)\nabla^f_{\bar E_j}\bar E_j,\bar E_i)- g(\nabla^f_{\bar E_j}[A_f(\xi)\bar E_j],\bar E_i)
& = -g([\nabla A_f](\bar E_j,\xi)\bar E_j,\bar E_i)\\
& = -g( [\nabla  \II_f](\bar E_j,\bar E_j,\bar E_i),\xi)\\
& =-g([\nabla  \II_f](\bar E_i,\bar E_j,\bar E_j),\xi).
\end{align*}
The last equality holds 
since $\nabla\II_f$ is totally symmetric in spaces of constant sectional curvature (see, e.g., \cite[Cor.\ 4.4]{KobNII}).
This proves \eqref{eq:Hi}.

Finally, we obtain \eqref{eq:Hbeta} directly from \eqref{eq:Hbeta1} using \eqref{eq:D_X}:
\[
H_\mu^\beta  = g(\sum_j\nabla^f_{\bar E_j}\bar E_j,d\pi(JE_\beta))   = g(\sum_j \II_f(\bar E_j,\bar E_j),\bar E_\beta).
\]
\end{proof}
\begin{rem}\label{rem:hypersurf}
The calculations in the previous proof lead quickly to Palmer's formula \cite[Prop.\ 3.4]{Pal97} for $H_\gamma$ when 
$f:M\to S^n$ is a hypersurface. We fix a unit normal field $\xi$ and work only with the connected
component of $UM^\perp$ it determines, so that $\mu:M\to US^n$, $\mu(p)=\xi_p$.  In this case the adapted local frame for
$\mu^{-1}\caC$ consists only of $\{E_i,JE_i:i=1,\ldots,n-1\}$. We can choose these so that about $p$ the frame $\{\bar
E_i\}$ for $T_pM$ is principal, i.e, it diagonalises the shape operator. We write $\caA_f\bar E_i = \kappa_i \bar E_i$,
and therefore
\[
E_i = \bar E_i-\kappa_i J\bar E_i,
\]
where we are identifying $\bar E_i$ with its horizontal lift in $\caH_\xi$. Thus $\|\bar E_i\|^2 = 1/(1+\kappa_i^2)$.
The calculations in the previous proof show that
\begin{align*}
h(H_\mu, JE_i) & = -\sum_j g([\nabla^f_{\bar E_i}A_f]\bar E_j,\bar E_j)\\
& =- \sum_j g(d\kappa_j(E_i)\bar E_j,\bar E_j)\\
& =- \sum_j\frac{1}{1+\kappa_j^2} d\kappa_j(E_i). 
\end{align*}
This means that the mean curvature $1$-form $\sigma_{H_\gamma} = \lambda_\caQ(H_\gamma,\cdot)$ of $\gamma$ can be written as
\begin{equation}\label{eq:sigmaH}
\sigma_{H_\gamma} = \sum_j\frac{1}{1+\kappa_j^2} d\kappa_j.
\end{equation}
\end{rem}

\section{Hamiltonian variations.}

We now consider smooth deformations $f_t:M\times I \to S^n$ of an immersion $f_0:M\to S^n$ of a compact manifold $M$
(without boundary).  Here $I\subset\R$ is an open interval
about $0$ and we assume that, for each $t$, $f_t$ is an immersion whose normal bundle is diffeomorphic to $TM^\perp$ (i.e.,
the normal bundle of $f_0$). 
As a consequence of the Covering Homotopy Theorem (specifically, \cite[Thm 11.4]{Ste}) there exists a 
compatible smooth deformation $\mu_t:UM^\perp\times I\to US^n$ of 
the spherical Gauss map $\mu_0$ of $f_0$, i.e., $\pi\circ\mu_t = f_t\circ\pi^\perp$ at each $t$. This
will only be uniquely determined when $M$ is a hypersurface. Otherwise any two such deformations differ by a 
smooth family of bundle
isomorphisms of $UM^\perp$, each covering the identity on $M$ (and equal to the identity on $UM^\perp$ at $t=0$). 
In any case, $\gamma_t = \pi_\caQ\circ\mu_t$ is a deformation of the geodesic Gauss map $\gamma_0$ of
$f_0$, through Lagrangian submanifolds.  Let $V_t\in \Gamma(\gamma_t^{-1}T\caQ)$ be 
the variational vector field corresponding to this, and let 
\[
\sigma_{V_t} = \lambda_\caQ(V_t,d\gamma_t)\in\Gamma(T^*(UM^\perp)),
\]
be the $1$-form corresponding to it.  Since $V_t$ is a Lagrangian variation, $\sigma_{V_t}$ is closed. 
Recall that $\gamma_t$ is called a \emph{Hamiltonian} deformation when $\sigma_{V_t}$ is also exact. 
We will prove two results about Lagrangian variations of $\gamma$. Each relies on the following facts about $US^n$
as a circle bundle over $\caQ$. 

For simplicity we will denote $US^n$ as a circle bundle by $\pi_\caQ:\caS\to\caQ$. The contact distribution
$\caC$ provides a principal bundle connexion, with connexion $1$-form $i\theta$ and curvature $-i\lambda$. Its 
first Chern class $c_1(\caS)$  is therefore the cohomology class of $\tfrac{1}{2\pi}\lambda_\caQ$. Consequently, the
pullback of $\caS$ to any Lagrangian immersion is a flat bundle.
Let $\caL\to\caQ$ be the tautological $2$-plane bundle of $\Gr(2,n+1)$ and identify $T\caQ$ with $\Hom(\caL,\caL^\perp)$. 
The complex structure $J_\caQ$ coincides with giving $\caL$ the Hermitian line bundle structure which identifies
its unit circle bundle with $\caS^{-1}$ (this is because geodesic flow is a \emph{right} action by $S$).
There is an isomorphism of Hermitian connected lines bundles $\caL^{n-1}\simeq K_\caQ$, where the latter is 
the canonical bundle. 
The curvatures are related via the \Kah-Einstein equation $i\rho_\caQ = (n-1)i\lambda_\caQ$, where $\rho_\caQ$ is
the Ricci curvature of $\caQ$.
\begin{prop}
For every smooth deformation $f_t$ of $f_0$ preserving the diffeomorphism class of the unit normal bundle, every compatible
deformation $\gamma_t$ of $\gamma_0$ is Hamiltonian.
\end{prop}
For hypersurfaces this was proved by Ma \& Ohnita \cite{MaO09}, where they also proved the converse for short term
deformations of $\gamma_t$. We will explain below why this converse will not generally hold for immmersions
$f:M\to S^n$ which have codimension at least two.
\begin{proof}
By \cite[Thm.\ 1.1]{MaO09}, for any Lagrangian variation $V$ of a Lagrangian
immersion $\gamma_0:UM^\perp\to\caQ$, $\sigma_{V_t}$ measures the variation in holonomy of the flat bundle 
$\caS=\gamma_0^{-1}US^n$, in
the following sense. Each flat bundle $\caS_t=\gamma_t^{-1}US^n$ determines a curve in the moduli space $H^1(U,S^1)$ of flat
bundles, whose tangent at $\caS_t$ is the cohomology class of $-i\sigma_{V_t}$. In particular, the family $\caS_t$ admits
a smooth family of horizontal sections $\mu_t$ if and only if $\sigma_{V_t}$ is exact. 
So if $\gamma_t$ arises from a deformation $f_t$ each $\caS_t$ has a horizontal section given by the spherical Gauss map
$\mu_t$, and therefore $\sigma_{V_t}$ is exact.
\end{proof}
In \cite{Pal97}, Palmer deduced from \eqref{eq:sigmaH} that the mean curvature $1$-form $\sigma_{H_\gamma}$ of
$\gamma$ is exact when $f$ is a hypersurface, so that mean curvature variation is Hamiltonian. We can prove this 
for the geodesic Gauss map of any
immersion $f:M\to S^n$ using another simple circle bundle argument.
\begin{prop}
For any immersion $f:M\to S^n$ with geodesic Gauss map $\gamma:UM^\perp\to\caQ$, 
the mean curvature vector field $H_\gamma$ is a Hamiltonian variation.
\end{prop}
An immediate corollary is that being Hamiltonian stationary and being Lagrangian stationary mean the same thing for
geodesic Gauss maps.
This also makes it reasonable to ask under what circumstances $H_\gamma$ arises from a deformation of $f$, for
it would be interesting to know when Lagrangian mean curvature flow of $\gamma$ keeps it inside the class of
geodesic Gauss maps of immersions of $M$. 
\begin{proof}
By \cite[Prop.\ 2.2]{Oh94}, $i\sigma_{H_\gamma}$ is the connexion $1$-form for the bundle $\gamma^{-1}K_\caQ$, the
pull-back of the canonical bundle over $\caQ$ (and by Dazord's fomula $d\sigma_{H_\gamma} = \gamma^*\rho_\caQ=0$).
Now $K_\caQ\simeq\caL^{n-1}$ as Hermitian connected line bundles
and $\gamma^{-1}\caL$ is trivial (it has horizontal section $\mu$). Thus $\gamma^{-1}K_\caQ$ is also trivial and
thus $\sigma_{H_\gamma}$ must be exact.
\end{proof}
\begin{rem}
If $f:M\to S^n$ is an immersion of codimension at least two, one cannot expect every Hamiltonian variation of its
geodesic Gauss map $\gamma$ to provide a deformation of $f$, even for a short time. There are immersions $f$
which admit smooth Lagrangian deformations $\gamma_t:UM^\perp\to\caQ$, $\gamma_0=\gamma$, for which $\gamma_t$ for $t\neq
0$ is the geodesic Gauss map of an immersion of $UM^\perp$ as a hypersurface in $S^n$. Such families are easy to
construct for $n=3$, by taking a family of immersions $f_t:S^1\times S^1\to S^3$ for which $d\mu_t$ is transverse to the
vertical distribution $\caV$ except at $t=0$, where one of the $S^1$ factors becomes vertical and thus $\mu_0$ is the
unit normal bundle for a closed curve $f_0$ in $S^3$. For example, such a family arises by considering the pre-image
under the Hopf fibration $S^3\to S^2$ of the intersection of a plane $V\subset\R^3$ with $S^2$. By tilting the plane
appropriately one can obtain a family of circles in $S^2$ which degenerate to a point at $t=0$. 
Lying over this on $S^3$ we have a a family of tori which degenerates to a geodesic at $t=0$. 

When $f$ is a hypersurface it is the stability of the transversal intersection $T(UM^\perp)\cap\caV\subset\caC$
which prevents degeneration from happening at least for short term deformations. Indeed, we expect that for \emph{every} 
immersion $f:M\to S^n$ of codimension at least two, there is a Hamiltonian deformation 
$\gamma_t:UM^\perp\to\caQ$ of its geodesic Gauss map for which $\pi\circ\mu_t:UM^\perp\to S^n$ is an immersion except 
at $t=0$. Of course, the
lifts $\mu_t:UM^\perp\to US^n$ always exist because the bundles $\gamma_t^{-1}\caS$ admit horizontal sections for
Hamiltonian deformations. 
\end{rem}

\end{document}